\documentclass[a4paper,11pt]{article}
\usepackage[utf8x]{inputenc}
\usepackage{amsmath, amsfonts,amssymb,mathtools,amsthm,mathrsfs,bbm}
\usepackage{bbding, pifont}
\usepackage{graphicx}
\usepackage{xcolor,a4wide}
\usepackage{bbding}
\usepackage[rightcaption]{sidecap}
\usepackage{csquotes}
\usepackage{tikz}
\definecolor{darkgrey}{rgb}{0.75,0.75,0.75}

\newcommand{\cadlag}{c\`adl\`ag}
\usepackage[rightcaption]{sidecap}
\newtheorem{theorem}{Theorem}
\newtheorem{proposition}{Proposition}[section]
\newtheorem{lemma}[proposition]{Lemma}

\newtheorem{definition}[proposition]{Definition}
\theoremstyle{definition}
\newtheorem{remark}[proposition]{Remark}
\newtheorem{example}[proposition]{Example}

\newcommand\unnumberedfootnote[1]{ %
        \let\temp=\thefootnote %
        \renewcommand{\thefootnote}{}%
        \footnote{#1}%
        \let\thefootnote=\temp%
        \addtocounter{footnote}{-1}}

\makeatletter
\renewcommand{\@fnsymbol}[1]{\ensuremath{%
   \ifcase#1\or 1\or 2\or 3\or
   \mathsection\or \mathparagraph\or \|\or 1\or
   2\or 3 \else\@ctrerr\fi}}
\makeatother

\numberwithin{equation}{section}

\begin{document}
\title{Some limit results for Markov chains indexed by trees}
\author{Peter Czuppon, Peter Pfaffelhuber} \date{\today}

\maketitle

\begin{abstract}
  \noindent
  We consider a sequence of Markov chains $(\mathcal X^n)_{n=1,2,...}$
  with $\mathcal X^n = (X^n_\sigma)_{\sigma\in\mathcal T}$, indexed by
  the full binary tree $\mathcal T = \mathcal T_0 \cup \mathcal T_1
  \cup ...$, where $\mathcal T_k$ is the $k$th generation of $\mathcal
  T$. In addition, let $(\Sigma_k)_{k=0,1,2,...}$ be a random walk on
  $\mathcal T$ with $\Sigma_k \in \mathcal T_k$ and
  $\widetilde{\mathcal R}^n = (\widetilde R_t^n)_{t\geq 0}$ with
  $\widetilde R_t^n := X_{\Sigma_{[tn]}}$, arising by
  observing the Markov chain $\mathcal X^n$ along the random walk. We
  present a law of large numbers concerning the empirical measure
  process $\widetilde{\mathcal Z}^n = (\widetilde Z_t^n)_{t\geq 0}$
  where $\widetilde{Z}_t^n = \sum_{\sigma\in\mathcal T_{[tn]}}
  \delta_{X_\sigma^n}$ as $n\to\infty$. Precisely, we show that if
  $\widetilde{\mathcal R}^n \xRightarrow{n\to\infty} \mathcal R$ for some Feller
  process $\mathcal R = (R_t)_{t\geq 0}$ with deterministic initial
  condition, then $\widetilde{\mathcal Z}^n \xRightarrow{n\to\infty} \mathcal Z$
  with $Z_t = \delta_{\mathcal L(R_t)}$.
\end{abstract}

\unnumberedfootnote{Keywords: Tree-indexed Markov chain, weak
  convergence, tightness, random measure, empirical measure}

\unnumberedfootnote{AMS Subject classification: 60F15;
  60F05} 

\section{Introduction}
In \cite{benjamini}, Benjamini and Peres introduced the notion of a
tree-indexed Markov chain. Since then, a lot of effort has been spent
in studying weak and strong laws of large numbers for very general
types of and even possibly random trees
\cite{liu2003,liu2004,yang2003,yang2006,takacs,Guyon2007}. 

Our work is motivated by an observation in microbiology, where a
population of bacteria is growing (along a binary tree, say), and
every individual bacterial cell is in a certain {\it state} (e.g.\
some gene expression profile), which can be -- atleast partially --
inherited. It has been observed for a long time that such populations
tend to be heterogeneous although all cells carry the same genome; see
\cite{SpudichKoshland1976} for an early reference.

The question which has arisen is about the mechanisms which are
responsible for such phenotypic heterogeneity. Two competing views
exist: either, random fluctuations lead to heterogeoity
\cite{McAdamsArkin1999,Elowitz2002} or social interactions of cells
together with a regulatory mechanism are key drivers for heterogeneity
\cite{SnijderPelkmans2011,Perlkmans2012}. Several examples are today
known to fall in one of the two categories; see the Review
\cite{Avery2006}.

In this manuscript, we analyse one consequence of the first view,
i.e.\ a law of large numbers. This results entails that the dynamics
of single cells can be stochastic while the behavior of the whole
population becomes deterministic. We will define a Markov kernel
dependent on some scaling parameter $n$ (which will tend to infinity)
and look at the empirical measure process in the $[nt]$-th generation
of the population, $t\geq 0$, which corresponds to a time-scaling of
the process of empirical measures. We will prove the weak convergence
of the empirical measure process, which will be a deterministic limit
(if the initial distribution is deterministic).

After presenting the general setup in Section~\ref{S:setup}, we present
our main result in Theorem~\ref{T1} in Section~\ref{S:results},
together with two simple examples. Then, we give the proof of
Theorem~\ref{T1} in Section~\ref{S:proof}.


\section{Setup}
\label{S:setup}
Let 
$$\mathcal T = \bigcup_{k=0}^\infty \mathcal T_k,\qquad 
\mathcal T_0 = \{\emptyset\}, \qquad \mathcal T_k = \{0,1\}^k,
k=1,2,...$$ be a \emph{complete binary tree}, rooted at $\emptyset \in
\mathcal T_0$, where $\sigma 0, \sigma 1 \in \mathcal T_{k+1}$ are the
two children of $\sigma \in \mathcal T_k$, $k=0,1,2,...$ For $\sigma
\in \mathcal T_k$ and $j\leq k$, we denote by $\pi_j \sigma$ the
prefix of $\sigma$ of length $j$. On $\mathcal T$, we set $|\sigma| =
k$ iff $\sigma\in\mathcal T_k$ and in addition, set $\pi_{-1} \sigma
:= \pi_{|\sigma|-1}\sigma$, the immediate ancestor of $\sigma$. Define
the $\leq$-relation by writing
$$ \sigma \leq \tau \quad \text{ iff }\quad \text{there is $j$ such that }
\pi_j \tau = \sigma$$ and
$$ \tau \wedge\tau' := \sup\{\sigma: \sigma \leq \tau, \sigma\leq \tau'\}$$
as the \emph{most recent common ancestor} of $\tau$ and $\tau'$.

~

Let $(E,r)$ be a complete and separable metric space, and denote by
$\mathcal B(E)$ the Borel-$\sigma$-field, or the set of bounded
measurable functions (with an abuse of notations). A stochastic
process $\mathcal X = (X_\sigma)_{\sigma\in\mathcal T}$ is called a
\emph{time-homogeneous, tree-indexed Markov chain} (extending a notion
introduced in \cite{benjamini}), if there is a Markov transition
kernel $p$ from $E$ to $\mathcal B(E^2)$ (the Borel-$\sigma$-field on
$E^2$) such that for all $\sigma \in \mathcal T$ and $A_0, A_1 \in
\mathcal B(E)$,
\begin{align*}
  \mathbf P(X_{\sigma 0} \in A_0, X_{\sigma 1} \in A_1 & | X_{\tau} =
  x_\tau \text{ for } \tau\in\mathcal T \text{ with } \tau\wedge
  \sigma \leq \sigma) \\ & = \mathbf P(X_{\sigma 0} \in A_0, X_{\sigma
    1} \in A_1 | X_{\sigma} = x_{\sigma}) = p(x_{\sigma}, A_0\times
  A_1).
\end{align*}
With $\mathcal X$, we connect the Markov chain $\mathcal
R = (R_n)_{n=0,1,2,...}$, with transition kernel
\begin{align*}
  p_{\mathcal R}(x,A) & := 
  \tfrac 12 (p(x,A\times E) + p(x,E\times A)).
\end{align*}
Here, $\mathcal R$ arises from observing the state of
$\mathcal X$ when walking along $\mathcal T$ starting from the root
from $\sigma$ to $\sigma 0$ and $\sigma 1$ purely at random. Another representation of $\mathcal R$ 
is as follows: Let
$(\Sigma_k)_{k=0,1,2,...}$ be a symmetric random walk on $\mathcal T$
(independent of $\mathcal X$), i.e.\ $\Sigma_k \in \mathcal T_k$
almost surely and $\mathbf P(\Sigma_{k+1} = \sigma0|\Sigma_k=\sigma) =
\mathbf P(\Sigma_{k+1} = \sigma1|\Sigma_k=\sigma) = \tfrac 12$. Then,
$\mathcal R \stackrel d = (X_{\Sigma_k})_{k=0,1,2,...}$.

If $(X_{\sigma})_{\sigma\in\mathcal T}$ is a (time-homogeneous) Markov
chain, we then define the \emph{process of empirical measures}
$\mathcal Z = (Z_k)_{k=0,1,2,...}$ through
$$ Z_k := 2^{-k} \sum_{\sigma \in \mathcal T_k} \delta_{X_\sigma}.$$
Note that $\mathcal Z$ takes values in $\mathcal P(E)$, the set of
probability measures on $\mathcal B(E)$ and that $\mathcal Z$ is a
non-homogeneous Markov chain (indexed by $k=0,1,2,...$).

\begin{remark}[Symmetric, tree-indexed Markov chains]

  The idea to consider different transition mechanisms to the two
  different children comes from the work of \cite{Guyon2007}. A
  special, classical case is taht of a symmetric tree-indexed Markov
  chain as follows:

  We call a time-homogeneous, (tree-indexed) Markov chain with
  transition kernel $p$ (from $E$ to $\mathcal B(E^2)$)
  \emph{symmetric}, if there is a Markov transition kernel $q$ (from
  $E$ to $\mathcal B(E)$) such that for all $x\in E$, $A_0,
  A_1\in\mathcal B(E)$
  $$ p(x, A_0\times A_1) = q(x, A_0) \cdot q(x, A_1).$$
  In other words, the transitions from $X_\sigma$ to $X_{\sigma 0}$
  and to $X_{\sigma 1}$ are independent. In this case, we have that
  $R_k \stackrel d = X_\sigma$ for all $\sigma \in \Sigma_k$.
\end{remark}

~

\noindent
In the next section, we will deal with a sequence $(\mathcal
X^n)_{n=1,2,...}$ of tree-indexed Markov chains.

\section{Results}\label{S:results}
Now, we state our main limit theorem for the setup given in the last
section.  Therefore, let $(\mathcal X^n)_{n=1,2,...}$ be a sequence of
tree-indexed Markov chains with complete and separable metric state
spaces $(E^n,r^n)_{n=1,2,...}$. As a limiting state space, we have a
complete separable metric space $(E,r)$ and Borel-measurable maps
$\eta^n: E^n\to E$. 

Let $\mathcal R^n$ be the process of observing $\mathcal X^n$ when
moving randomly along the tree.  We denote the corresponding
transition kernel by $p_n$ (for $\mathcal X^n$) and $p_{\mathcal
  R^n}$, respectively. Moreover, let $\mathcal Z^n$ be the process of
empirical measures based on $\mathcal X^n$, which has state space
$\mathcal P(E^n)$, $n=1,2,...$ Our goal is to find sufficient
conditions for $\mathcal X^n$ (via $\mathcal R^n$), such that the
process of empirical measures $\mathcal Z^n$ converges, and to
characterize the limit process. We first recall some basic notation.

\begin{remark}[Notation]
  Throughout the manuscript, we will consider a complete and separable
  metric space $(E,r)$. The space of (continuous,) real-valued,
  bounded functions on $E$ are denoted by $\mathcal B(E) (\mathcal
  C_b(E))$. Weak convergence is denoted by $\Rightarrow$.  If $f:
  [0,\infty) \to E_1$ and $\eta: E_1\to E_2$, we write, abusing
  notation, $\eta\circ f = \eta(f)$ for the function $\eta\circ f:
  t\mapsto\eta(f(t))$ If $(E_1,r_1), (E_2,r_2)$ are two metric spaces,
  $\eta: E_1\to E_2$ is measurable, and $\nu\in\mathcal P(E_1)$, we
  define the image measure of $\nu$ under $\eta$ by $\eta_\ast \nu \in
  \mathcal P(E_2)$, i.e.\ $\eta_\ast \nu(A_2) = \nu(\eta^{-1}(A_2))$.
  Sometimes, we write $\langle z, \varphi\rangle := \int \varphi dz$
  for $z\in\mathcal P(E)$ and $\varphi\in\mathcal B(E)$. For
  $f\in\mathcal C_b(E)$ we write $||f|| := \sup_{x\in E} |f(x)|$.
\end{remark}

\noindent
We need two more notions.

\begin{definition}[Feller property, compact containment
  condition\label{def:compCont}]
  Recall that $(E,r)$ is complete and separable.
  \begin{enumerate}
  \item A Markov process $\mathcal X = (X_t)_{t\geq 0}$ with state
    space $E$ and \cadlag\ paths satisfies the \emph{Feller property},
    iff (i) $X_t\xRightarrow{t\to 0} X_0$ and (ii) the map $x\mapsto
    \mathbf E[f(X_t)|X_0=x]$ is continuous for all $f\in\mathcal
    C_b(E)$, $t\geq 0$ and all $x\in E$. Equivalently, let
    $(S_t)_{t\geq 0}$ be the semigroup of $\mathcal X$, i.e.\ $S_tf(x)
    = \mathbf E[f(X_t)|X_0=x]$.  Then, $\mathcal X$ is a
    Feller-process iff $(S_t)_{t\geq 0}$ is a Feller semigroup, i.e.\
    (i) $S_tf(x) \xrightarrow{t\to 0}f(x)$ for all $x\in E$ and
    $f\in\mathcal C_b(E)$ and (ii) $S_tf \in \mathcal C_b(E)$ if
    $f\in\mathcal C_b(E)$. We say that (iii) $(S_t)_{t\geq 0}$ is a
    contraction iff $||S_tf|| \leq ||f||$ and (iv) $(S_t)_{t\geq 0}$
    is strongly continuous iff
    $||S_tf - f|| \xrightarrow{t\to 0} 0$.\\
    We say that an operator $G_{\mathcal X}: \mathcal D(G_{\mathcal
      X}) \subseteq \mathcal C_b(E) \to \mathcal C_b(E)$ generates a
    strongly continuous semigroup $(S_t)_{t\geq 0}$ if
    \begin{align}\label{eq:gener}
      G_{\mathcal X} f(x) := \lim_{t\to 0} \tfrac 1t (S_tf(x) - f(x))
    \end{align}
    for all $f\in\mathcal C_b(E)$ for which the limit
    in~\eqref{eq:gener} exists.\\
    Recall that if $(E,r)$ is locally compact, every Feller semigroup
    is a strongly continuous contraction semigroup (\cite{kall},
    Theorem 17.6) and is uniquely characterized by its generator
    (\cite{kall}, Lemma 17.5)
  \item For a sequence $(X^1_t)_{t\geq 0}, (X^2_t)_{t\geq 0},...$ of
    $E$-valued stochastic processes, we say that the \emph{compact
      containment condition} (in $E$) holds, if for every
    $\varepsilon>0$ and $T\geq 0$ there is a compact set
    $K_{\varepsilon, T}\subseteq E$ such that
    \begin{align*}
      \sup_{n=1,2,...} \mathbf P(X_t^n \in K_{\varepsilon, T}^c \text{
        for all } 0\leq t\leq T) < \varepsilon.
    \end{align*}
  \end{enumerate}
\end{definition}

\noindent
Now we can formulate our main result.

\begin{theorem}[Convergence of $\mathcal Z^n$\label{T1}]
  Let $\mathcal X^n, \mathcal R^n, \mathcal Z^n$ be as above,
  $n=1,2,...$ Moreover, let $\widetilde{\mathcal R}^n := (\widetilde
  R_{t}^n)_{t\geq 0}:= (R_{[nt]}^n)_{t\geq 0}$, and
  $\widetilde{\mathcal Z}^n := (\widetilde Z_{t}^n)_{t\geq 0}:=
  (Z_{[nt]}^n)_{t\geq 0}$, $n=1,2,...$ Assume that $\eta^n(X_0^n)
  \xRightarrow{n\to\infty} \nu \in\mathcal P(E)$ and that the compact
  containment condition holds for $\eta^1(\widetilde{\mathcal
    R}^1),\eta^2(\widetilde{\mathcal R}^1),...$

  In addition, assume that there is a linear operator $G_{\mathcal R}:
  \mathcal D(G_{\mathcal R}) \subseteq \mathcal C_b(E) \to \mathcal
  C_b(E)$, which generates a strongly continuous contraction
  semigroup, and such that $\mathcal D(G)$ contains an algebra $\Pi$
  that separates points. For each $\varphi\in\mathcal D(G_{\mathcal
    R})$, there is a sequence $\varphi_1 \in \mathcal B(E^1),
  \varphi_2 \in \mathcal B(E^2),...$ such that $\sup_{n=1,2,...}
  ||\varphi_n|| < \infty$ and
  \begin{align}\label{eq:T11}
      \lim_{n\to\infty}\sup_{x\in E^n} | \varphi\circ \eta^n(x) - \varphi_n(x)| & = 0,\\
      \lim_{n\to\infty}\sup_{x\in E^n} | (G_{\mathcal R} \varphi)\circ \eta^n(x) -
    G_{\widetilde{\mathcal R}^n}\varphi_n(x)| & = 0,\label{eq:T12}
  \end{align}
  where
  \begin{align*}
    G_{\widetilde{\mathcal R}^n}\varphi(x) := n \cdot \mathbf
    E[\varphi(\widetilde R^n_{1/n}) - \varphi(\widetilde
    R^n_{0})|\widetilde R^n_{0}=x].
  \end{align*}
  Then, there is an $E$-valued Feller process $\mathcal R =
  (R_t)_{t\geq 0}$ with $R_0\sim\nu$ and generator $G_{\mathcal R}$
  with $\eta(\widetilde{\mathcal R}^n) \xRightarrow{n\to\infty}
  \mathcal R$, and a $\mathcal P(E)$-valued stochastic process
  ${\mathcal Z} = (Z_t)_{t\geq 0}$ such that
  $\eta^n_\ast\widetilde{\mathcal Z}^n \xRightarrow{n\to \infty}
  {\mathcal Z}$ with $Z_0 \sim \delta_\nu \in\mathcal P(\mathcal
  P(E))$. Moreover, if $\nu = \delta_x$ for $x\in E$, then $Z_t =
  \delta_{\mathcal L(R_t)}$.
\end{theorem}

\begin{remark}[Convergence, Deterministic limit, CLT]
  \begin{enumerate}
  \item Actually, the convergence $\widetilde{\mathcal R}^n \circ\eta^n
    \xRightarrow{n\to\infty} \mathcal R$ was shown in \cite{ethier},
    Corollary~4.8.9, under the assumptions given above.
  \item As the Theorem shows, the limiting process of empirical
    measures $\mathcal Z$ is deterministic (if the initial
    distribution is a Dirac-measure). The heuristics behind this
    result is that two distinct values $X^n_\sigma, X^n_\tau$ with
    $\sigma,\tau \in \mathcal T_{[nt]}$ have already evolved
    independently for $O(n)$ steps. Hence, $\widetilde{Z}_t^n$ is
    approximately given by the empirical measure of $2^{nt}$
    independent processes, which leads to a deterministic limit. This
    argument will be made precise below.
  \item Having obtained a law of large numbers, it would be
    interesting to see a central limit theorem, as well. In the
    present context, this would require a fine analysis of the error
    terms $\varepsilon_n$ appearing in \eqref{eq:113}. We devote this
    study to future research.
  \end{enumerate}
\end{remark}

\noindent
We now give two simple examples for normal and Poisson convergence.

\begin{example}
\begin{enumerate}
\item Let $(Y_\sigma)_{\sigma\in\mathcal T}$ be a family of
  independent, identically distributed random real-valued variables
  with $\mathbf E[Y_\sigma]=0, \mathbf{Var}[Y_\sigma]=1$. Moreover,
  let $X_0^n:=0$ and $(X^n_{\sigma 0}, X^n_{\sigma 1}) :=
  \big(X^n_\sigma + \tfrac{1}{\sqrt n}Y_\sigma, X_\sigma -
  \tfrac{1}{\sqrt n}Y_\sigma\big)$. (In other words, the states of the
  two children of $\sigma$ are a pair of dependent random variables.)
  Then, the process $\mathcal R^n = (R_t^n)_{t=0,1,2,...}$ can be
  written as $R_t^n \stackrel d = \tfrac{1}{\sqrt{n}}\sum_{k=0}^{t-1}
  \widetilde Y_k$, where $(\widetilde Y_k)_{k=0,1,2,...}$ are
  independent and identically distributed with $(\widetilde Y_k)_\ast
  \mathbf P = \tfrac 12 (Y_\sigma)_\ast \mathbf P + \tfrac 12
  (-Y_\sigma)_\ast \mathbf P$, a mixture of the distributions of
  $Y_\sigma$ and $-Y_\sigma$.  Donsker's Theorem yields the
  convergence $\widetilde{\mathcal{R}}^n\xRightarrow{n\to\infty}
  \mathcal B$ to the standard Brownian motion $\mathcal B$. Our
  theorem now says that the limiting process $\mathcal{Z}$ is the law
  of $\mathcal B$, so we find that $\mathcal Z = (N(0,t))_{t\geq 0}$,
  where $N(0,t)$ is the normal distribution with mean $0$ and variance
  $t$.
\item Let $(Y^n_\sigma)_{\sigma\in\mathcal T}$ be a family of
  independent, identically distributed random variables with values in
  $\mathbb Z_+$ and $\mathbf P(Y^n_\sigma>0)=2\lambda/n + o(1/n),
  \mathbf{P}[Y^n_\sigma>1]=o(1/n)$. Moreover, let $X_0^n:=0$ and
  $(X^n_{\sigma 0}, X^n_{\sigma 1}) := \big(X_\sigma, X^n_\sigma +
  Y^n_\sigma\big)$. (In other words, the state of the left child
  equals the state of its parent while the state of the right child
  has a small probability of having increased by~1.  Then, the process
  $\mathcal R^n = (R_t^n)_{t=0,1,2,...}$ can be written as $R_t^n
  \stackrel d = \sum_{k=0}^{t-1} \widetilde Y^n_k$, where $(\widetilde
  Y^n_k)_{k=0,1,2,...}$ are independent and identically distributed
  with $(\widetilde Y^n_k)_\ast \mathbf P = \tfrac 12 \delta_0 +
  \tfrac 12 (Y^n_\sigma)_\ast \mathbf P$, i.e.\ $\mathbf P[\widetilde
  Y^n_k>0]=\lambda/n + o(1/n), \mathbf{P}[\widetilde Y^n_k>1]=o(1/n)$.
  Classical convergence results (see e.g.\ \cite{kall}, Theorem 5.7)
  then show that $\widetilde{\mathcal R}^n$ converges weakly to a
  Poisson process with rate $\lambda$. Consequently, we then have by
  the above theorem that $\mathcal Z = (Z_t)_{t\geq 0}$ with $Z_t =
  \text{Poi}(\lambda t)$.
\end{enumerate}
\end{example}
 
\section{Proof of Theorem~\ref{T1}}
\label{S:proof}
Throughout this section, we build on the same assumptions as in
Theorem~\ref{T1}. We will replace $\eta^n(\widetilde{\mathcal R}^n)$
by $\widetilde{\mathcal R}^n$ and $\eta_\ast^n \widetilde{\mathcal
  Z}^n$ by $\widetilde{\mathcal Z}^n$ in the sequel (and similarly for
the processes without $\sim$). This should not cause confusion and
increase readability.

Before we start, we give basic relationships between the processes
$\widetilde{\mathcal R}^n$ and $\widetilde{\mathcal Z}^n$, which we
will frequently use. (Some more refined relationships will be given in
the proof of Lemma~\ref{l:convZ}. Let $\varphi \in \mathcal
C_b(E)$. Then,
\begin{equation}
  \label{eq:111}
  \begin{aligned}
    \mathbf E[\langle \widetilde Z_t^n, \varphi\rangle] & = \mathbf
    E[\langle Z_{[nt]}^n, \varphi\rangle] = \mathbf
    E\Big[\frac{1}{2^{[nt]}} \sum_{\sigma \in \mathcal T_{[nt]}}
    \langle \delta_{X_\sigma^n}, \varphi\rangle\Big] = \mathbf
    E\Big[\frac{1}{2^{[nt]}} \sum_{\sigma\in \mathcal T_{[nt]}}
    \varphi(X_\sigma^n)\Big] \\ & = \mathbf E[\varphi(R^n_{[nt]})] =
    \mathbf E[\varphi(\widetilde R^n_t)].
  \end{aligned}
\end{equation}
Similarly, we write
\begin{equation}
  \label{eq:112}
  \begin{aligned}
    \langle Z_k^n, \varphi\rangle & = \sum_{\sigma\in\mathcal T_k}
    \varphi(X_\sigma) = \mathbf E[\varphi(R_k^n)|Z_k^n],\\
    \mathbf E[\langle Z_k^n, \varphi\rangle| Z_{k-1}^n] & =
    \mathbf E\big[\mathbf E[\varphi(R_k^n)|Z_k^n,Z_{k-1}^n]|
    Z_{k-1}^n\big] = \mathbf E[\varphi(R_k^n)| Z_{k-1}^n].
  \end{aligned}
\end{equation}

~

\noindent
In the proof of Theorem~\ref{T1}, it suffices to assume that $\nu =
\delta_x$, i.e.\ deterministic starting conditions. (The general case
then follows by mixing over the initial condition.) We need to show
two assertions:
\begin{enumerate}
\item The sequence $(\widetilde{\mathcal Z}^n)_{n=1,2,...}$ is tight.
\item The finite-dimensional distributions of $(\widetilde{\mathcal
    Z}^n)_{n=1,2,...}$ converge, such that $\widetilde{Z}_t^n
  \xRightarrow{n\to\infty} \delta_{\mathcal L(R_t)}$.
\end{enumerate}
For 2., we will show in Lemma~\ref{l:convZ} that $\widetilde{Z}_t^n
\xRightarrow{n\to\infty} \delta_{\mathcal L(R_t)}$ holds for all
$t\geq 0$. Since the right hand side is deterministic, we have already
shown convergence of finite dimensional distribution and we are left
with showing~1. Here, we use Jakubowski's tightness criterion, which
is recalled in Proposition~\ref{P:jaku} in the appendix. For this
criterion, we have to show that (i) $\widetilde{\mathcal Z}^n$
satisfies the compact containment condition in $\mathcal P(E)$ (see
Definition~\ref{def:compCont}) and (ii) that the sequence $(\langle
\widetilde Z_t^1, \varphi\rangle)_{t\geq 0}$, $(\langle \widetilde
Z_t^2, \varphi\rangle)_{t\geq 0}$... is tight for all $\varphi\in\Pi'$
(a vector space which separates points). (i) will be resolved in
Lemma~\ref{l:CompCont}, while (ii) is a result in
Lemma~\ref{l3}. Hence, we are done once we have shown Lemma
~\ref{l:convZ}, ~\ref{l:CompCont} and~\ref{l3}. 

~

\noindent
We start with a fundamental fact, which is based on the fact that two
random leaves from $\mathcal T_n$ have a most recent common ancestor
node which is close to the root. \\
Recall that by \cite{ethier}, Corollary~4.8.9 we already have that
$\widetilde{\mathcal R}^n \xRightarrow{n\to\infty} \mathcal R$ for a
Feller- (hence \cadlag)-process $\mathcal R$.

\begin{lemma}[Convergence at fixed vertices\label{l1}]
  Assume that $\widetilde{\mathcal R}^n
  \xRightarrow{n\to\infty}\mathcal R$ for a \cadlag-process $\mathcal
  R = (R_t)_{t\geq 0}$ with state space $E$. Then, the following
  holds:
  \begin{enumerate}
  \item Let $\sigma_1,...,\sigma_k \in \mathcal T$. Then,
    $$ (X_{\sigma_i}^n)_{i=1,...,k} \xrightarrow{n\to\infty} (R_0)_{i=1,...,k}$$
    in probability.
  \item Let $\Sigma^n_1, \Sigma^n_2$ be two vertices, chosen uniformly
    at random from $\mathcal T_{[nt]}$. Then,
    \begin{align*}
      (X^n_{\Sigma^n_1\wedge\Sigma^n_2},
      X^n_{(\Sigma^n_1\wedge\Sigma^n_2)0},
      X^n_{(\Sigma^n_1\wedge\Sigma^n_2)1}) \xrightarrow{n\to\infty}
      (R_0, R_0, R_0)
    \end{align*}
    in probability.
  \end{enumerate}
\end{lemma}

\begin{proof}
  Recall that for the (independent) random walk
  $(\Sigma_k)_{k=0,1,...}$ on $\mathcal T$ we have that $R^n_k =
  X^n_{\Sigma_k}$. It suffices to prove the result for deterministic
  $R_0\in E$. By assumption, for all $m\in\mathbb N$,
  \begin{align}\label{eq:l11}
    \mathbf P(r(R_m^n, R_0)>\varepsilon) = \mathbf P(r(\widetilde
    R_{m/n}^n, R_0)>\varepsilon) \xrightarrow{n\to\infty} 0,
  \end{align}
  since $\mathcal R$ has \cadlag\ paths.

  \noindent
  1. Let $\sigma\in \mathcal T$ and $|\sigma|=m$. Assume that the
  assertion does not hold, i.e.\ $X_\sigma^n$ does not converge weakly
  to $R_0$. Let $\varepsilon>0$ such that $\mathbf P(r(X_\sigma^n,
  R_0)>\varepsilon) > \varepsilon$ for all $n$. We have that $$\mathbf
  P(r(R^n_m, X_\sigma^n)\leq \varepsilon) \geq \mathbf P(r(R^n_m,
  X_\sigma^n)\leq \varepsilon, R^n_m= X_\sigma^n) = \mathbf P(R^n_m=
  X_\sigma^n) \geq 2^{-m}$$ for all $\varepsilon>0$, since the random
  walk $(\Sigma_m)_{m=0,1,2,...}$ along we read off $R^n$ has a chance
  of $2^{-m}$ to pass through vertex $\sigma$. Hence, this implies
  that for $\varepsilon>0$ as above
  \begin{align*}
    \mathbf P(r(R_m^n, R_0)>\varepsilon) & \geq \mathbf P(r(R_m^n,
    R_0)>\varepsilon, R_m^n = X_\sigma^n) \geq \mathbf P(r(X_\sigma^n,
    R_0)>\varepsilon, \Sigma_m=\sigma) \\ & = \mathbf P(r(X_\sigma^n,
    R_0)>\varepsilon) \cdot \mathbf P(\Sigma_m=\sigma) \geq
    \varepsilon 2^{-m}
  \end{align*}
  in contradiction to~\eqref{eq:l11}. Hence, 1.\ follows.
  
  \noindent
  2. Let $\varepsilon>0$ and $m$ be large enough for $2^{-m} <
  2\varepsilon$. From 1., we have that
  $(X_\sigma^n)_{\sigma\in\mathcal T_m} \xrightarrow{n\to\infty}
  (R_0)_{\sigma\in\mathcal T_m}$. Moreover, for $n>m$, $\mathbf
  P(\Sigma_1^n \wedge \Sigma_2^n \in \mathcal T_m) = \sum_{k=0}^m
  2^{-(k+1)} = 1 - 2^{-(m+1)} > 1 - \varepsilon$. Hence, we can write
  \begin{align*}
    \mathbf P(r(X_{\Sigma_1^n \wedge \Sigma_2^n}, R_0)>\varepsilon) &
    \leq \mathbf P(r(X_{\Sigma_1^n \wedge \Sigma_2^n},
    R_0)>\varepsilon, \Sigma_1^n \wedge \Sigma_2^n \in \mathcal T_m)+
    \mathbf P(\Sigma_1^n \wedge \Sigma_2^n \notin \mathcal T_m) \\ &
    \leq \mathbf P(\sup_{\sigma\in \mathcal T_m} r(X_\sigma^n,
    R_0)>\varepsilon) + \mathbf P(\Sigma_1^n \wedge \Sigma_2^n \notin
    \mathcal T_m) \\ & \xrightarrow{n\to\infty} 2^{-(m+1)} <
    \varepsilon
  \end{align*}
  by 1.\ and we have shown that $X_{\Sigma_1^n\wedge \Sigma_2^n}
  \xrightarrow{n\to\infty} R_0$ in probability. By the same arguments,
  we also find that $X_{(\Sigma_1^n\wedge \Sigma_2^n)i}
  \xrightarrow{n\to\infty} R_0$ in probability for $i=0,1$ and we are
  done.
\end{proof}

\begin{lemma}[Convergence of $\widetilde{\mathcal Z}^n$ at fixed
  times\label{l:convZ}]
  Consider the same situation as in Theorem~\ref{T1} and let $t\geq
  0$. If $\nu=\delta_x$ for some $x\in E$, we have that $\widetilde
  Z_t^n\xRightarrow{n\to\infty} \delta_{\mathcal L(\widetilde R_t)}$.
\end{lemma}

\begin{proof}
  Note that the assertion holds once we show that 
  \begin{align}
    \label{eq:Z1}
    \langle Z_t^n, \varphi\rangle \xRightarrow{n\to\infty} \mathbf
    E[\varphi(\widetilde R_t)]
  \end{align}
  for all $\varphi \in \mathcal C_b(E)$. (Indeed, the family $(\langle
  Z_t^n, \varphi\rangle)_{n=1,2,...}$ is tight by the boundedness of
  $\varphi$ and any subsequent limit point is deterministic by
  Lemma~\ref{l:detMM}.) For this, we already know from~\eqref{eq:111}
  that $\mathbf E[\langle Z_t^n, \varphi\rangle] = \mathbf
  E[\varphi(\widetilde R_t^n)] \xrightarrow{n\to\infty} \mathbf
  E[\varphi(\widetilde R_t)]$. Further we will show that
  \begin{align}
    \label{eq:112b}
    \mathbf{Var}[\langle Z_t^n, \varphi\rangle]
    \xrightarrow{n\to\infty} 0
  \end{align}
  which then implies \eqref{eq:Z1}. For this, consider two randomly
  picked vertices $\Sigma_1, \Sigma_2 \in \mathcal T_{[nt]}$ with
  $\Sigma_1\neq\Sigma_2$. Then, without loss of generality we assume
  that $\pi_{|\Sigma_1\wedge\Sigma_2|+1} X_{\Sigma_1} =
  X_{(\Sigma_1\wedge\Sigma_2)0}$ and $\pi_{|\Sigma_1\wedge\Sigma_2|+1}
  X_{\Sigma_2} = X_{(\Sigma_1\wedge\Sigma_2)1}$ such that
  \begin{align*}
    \mathbf E[\langle \widetilde Z_t^n, & \varphi\rangle^2] =
    \frac{1}{2^{2[nt]}} \sum_{\sigma_1, \sigma_2 \in \mathcal
      T_{[nt]}} \mathbf E[\varphi(X^n_{\sigma_1})
    \varphi(X^n_{\sigma_2})] \\ & = \mathbf E[\varphi(X^n_{\Sigma_1})
    \varphi(X^n_{\Sigma_2})] + \frac{1}{2^{[nt]}} \big(\mathbf
    E[\varphi^2(X_{\Sigma_1})-\varphi(X_{\Sigma_1})\varphi(X_{\Sigma_2})]\big)\\
    & = \mathbf E\big[\mathbf E[\varphi(X^n_{\Sigma_1})
    \varphi(X^n_{\Sigma_2})|X^n_{(\Sigma^n_1\wedge \Sigma_2^n)0},
    X^n_{(\Sigma^n_1\wedge \Sigma_2^n)1}]\big] +
    \frac{1}{2^{[nt]}}\big(\mathbf
    E[\varphi^2(X_{\Sigma_1})-\varphi(X_{\Sigma_1})\varphi(X_{\Sigma_2})]\big)
    \\ & = \mathbf E\big[\mathbf E[\varphi(X^n_{\Sigma_1})
    |X^n_{\pi_{|\Sigma^n_1\wedge \Sigma_2^n|+1}\Sigma_1}]\cdot \mathbf
    E[\varphi(X^n_{\Sigma_2}) |X^n_{\pi_{|\Sigma^n_1\wedge
        \Sigma_2^n|+1}\Sigma_2}]\big] \\ & \qquad \qquad \qquad \qquad
    \qquad \qquad \qquad \qquad \qquad \qquad + \frac{1}{2^{[nt]}}
    \big(\mathbf
    E[\varphi^2(X_{\Sigma_1})-\varphi(X_{\Sigma_1})\varphi(X_{\Sigma_2})]\big)\\
    & = \mathbf E\big[ \mathbf E[\varphi(R^n_{[nt] -
      |\Sigma_1\wedge\Sigma_2|-1})| R_0^n =
    X^n_{\pi_{|\Sigma^n_1\wedge \Sigma_2^n|+1}\Sigma_1}] \\ & \qquad
    \cdot \mathbf E[\varphi(R^n_{[nt] - |\Sigma_1\wedge\Sigma_2|-1})|
    R_0^n = X^n_{\pi_{|\Sigma^n_1\wedge \Sigma_2^n|+1}\Sigma_2}]\big]
    + \frac{1}{2^{[nt]}} \big(\mathbf
    E[\varphi^2(X_{\Sigma_1})-\varphi(X_{\Sigma_1})\varphi(X_{\Sigma_2})]\big)\\
    & = \mathbf E[\varphi(\widetilde R_t^n)]^2 + \varepsilon_t^n =
    \mathbf E[\langle \widetilde Z_t^n, \varphi\rangle]^2 +
    \varepsilon_t^n
  \end{align*}
  for 
  \begin{equation}
    \label{eq:113}
    \begin{aligned}
      \varepsilon_t^n & := \mathbf E\big[ \mathbf E[\varphi(R^n_{[nt]
        - |\Sigma_1\wedge\Sigma_2|-1}) - \mathbf
      E[\varphi(R_{[nt]}^n)]| R_0^n = X^n_{\pi_{|\Sigma^n_1\wedge
          \Sigma_2^n|+1}\Sigma_1}] \\ & \quad \qquad \qquad \qquad
      \cdot \mathbf E[\varphi(R^n_{[nt] -
        |\Sigma_1\wedge\Sigma_2|-1})- \mathbf E[\varphi(R_{[nt]}^n)]|
      R_0^n = X^n_{\pi_{|\Sigma^n_1\wedge \Sigma_2^n|+1}\Sigma_2}]\big] \\
      & \quad + \mathbf E[\varphi(R_{[nt]}^n)]\cdot \mathbf
      E[\varphi(R^n_{[nt] - |\Sigma_1\wedge\Sigma_2|-1})- \mathbf
      E[\varphi(R_{[nt]}^n)]| R_0^n = X^n_{\pi_{|\Sigma^n_1\wedge
          \Sigma_2^n|+1}\Sigma_1}]\big]\\ & \quad + \mathbf
      E[\varphi(R_{[nt]}^n)]\cdot \mathbf E[\varphi(R^n_{[nt] -
        |\Sigma_1\wedge\Sigma_2|-1})- \mathbf E[\varphi(R_{[nt]}^n)]|
      R_0^n = X^n_{\pi_{|\Sigma^n_1\wedge \Sigma_2^n|+1}\Sigma_2}]\big] \\
      & \qquad \qquad \qquad \qquad \qquad \qquad \qquad \qquad \qquad
      + \frac{1}{2^{[nt]}} \big(\mathbf
      E[\varphi^2(X_{\Sigma_1})-\varphi(X_{\Sigma_1})\varphi(X_{\Sigma_2})]\big)
      \\ & = \mathbf{COV}\big[\mathbf E[\varphi(R^n_{[nt] -
        |\Sigma_1\wedge\Sigma_2|-1})| R_0^n =
      X^n_{\pi_{|\Sigma^n_1\wedge \Sigma_2^n|+1}\Sigma_1}], \\ &
      \qquad \qquad \qquad \mathbf E[\varphi(R^n_{[nt] -
        |\Sigma_1\wedge\Sigma_2|-1})| R_0^n =
      X^n_{\pi_{|\Sigma^n_1\wedge \Sigma_2^n|+1}\Sigma_2}]] \\ & +
      2\cdot\mathbf E[\varphi(R_{[nt]}^n)]\cdot \mathbf
      E[\varphi(R^n_{[nt] - |\Sigma_1\wedge\Sigma_2|-1})- \mathbf
      E[\varphi(R_{[nt]}^n)]| R_0^n = X^n_{\pi_{|\Sigma^n_1\wedge
          \Sigma_2^n|+1}\Sigma_2}]\big] \\
      & \qquad \qquad \qquad \qquad \qquad \qquad \qquad \qquad \qquad
      + \frac{1}{2^{[nt]}} \big(\mathbf
      E[\varphi^2(X_{\Sigma_1})-\varphi(X_{\Sigma_1})\varphi(X_{\Sigma_2})]\big).
    \end{aligned}
  \end{equation}
  Hence, we must show $\varepsilon_t^n \xrightarrow{n\to\infty}0$ for
  \eqref{eq:112b}, which is implied by the boundedness of $\varphi$
  (showing convergence to~0 of the last term in the last line
  of~\eqref{eq:113}), by the Cauchy-Schwartz inequality and
  \begin{align}
    \label{eq:600}
    \mathbf E[\varphi(R^n_{[nt] - |\Sigma_1\wedge\Sigma_2|-1})|
    R_0^n = X^n_{\pi_{|\Sigma^n_1\wedge \Sigma_2^n|+1}\Sigma_1}]
    \xrightarrow{n\to\infty} \mathbf E[\varphi(R_{t})| R_0 = x]
  \end{align}
  in probability. We already know from Lemma~\ref{l1} that
  $X^n_{\pi_{|\Sigma^n_1\wedge \Sigma_2^n|+1}\Sigma_1}
  \xrightarrow{n\to\infty} x$ in probability, such that, since
  $\mathcal R$ has \cadlag\ paths, convergence of semigroups and
  \cite{ethier}, Theorem 1.6.1 (see also Remark 4.8.8) and the strong
  continuity of the semigroup for $\mathcal R$,
  \begin{align*}
    |\mathbf E[\varphi(& R^n_{[nt] - |\Sigma_1\wedge\Sigma_2|-1})|
    R_0^n = X^n_{\pi_{|\Sigma^n_1\wedge \Sigma_2^n|+1}\Sigma_1}]-
    \mathbf E[\varphi(R_{t})| R_0 = x]| \\ & \leq |\mathbf
    E[\varphi(R^n_{[nt] - |\Sigma_1\wedge\Sigma_2|-1})| R_0^n =
    X^n_{\pi_{|\Sigma^n_1\wedge \Sigma_2^n|+1}\Sigma_1}] - \mathbf
    E[\varphi(R^n_{[nt]}) | R_0^n = X^n_{\pi_{|\Sigma^n_1\wedge
        \Sigma_2^n|+1}\Sigma_1}]| \\ & + |\mathbf E[\varphi(\widetilde
    R^n_{t}) | \widetilde R_0^n = X^n_{\pi_{|\Sigma^n_1\wedge
        \Sigma_2^n|+1}\Sigma_1}]| - \mathbf E[\varphi(R_{t}) | R_0 =
    X^n_{\pi_{|\Sigma^n_1\wedge \Sigma_2^n|+1}\Sigma_1}]| \\ & +
    |\mathbf E[\varphi(R_{t}) | R_0 = X^n_{\pi_{|\Sigma^n_1\wedge
        \Sigma_2^n|+1}\Sigma_1}] - \mathbf E[\varphi(R_{t})| R_0 = x]|
    \\ & \xrightarrow{n\to\infty} 0
  \end{align*}
  in probability, which shows \eqref{eq:600}. This completes the proof.
\end{proof}

\noindent
Now, we come to the proof of the compact containment condition for
$(\widetilde{\mathcal Z}^n)_{n=1,2,...}$.

\begin{lemma}[Compact containment condition for $\widetilde{\mathcal
    Z}^n$\label{l:CompCont}]
  If $(\widetilde{\mathcal R}^n)_{n=1,2,...}$ satisfies the compact
  containment condition (in $E$), then $(\widetilde{\mathcal
    Z}^n)_{n=1,2,...}$ satisfies the compact containment condition (in
  $\mathcal P(E)$) as well.
\end{lemma}

\begin{proof}
  The proof is by contradiction. Assume that $(\widetilde{\mathcal
    R}^n)_{n=1,2,...}$ satisfies the compact containment condition,
  but the compact containment condition for $(\widetilde{\mathcal
    Z}^n)_{n=1,2,...}$ does not hold. Let $\varepsilon>0$ and $T\in
  \mathbb R_+$ be such that
  \begin{align} \label{eq:1} \sup_{n=1,2,...} \mathbf P(\widetilde
    Z_t^n \notin L \text{ for some }0\leq t\leq T) > \varepsilon
  \end{align}
  for all $L\subseteq \mathcal P(E)$ compact. (Such an $\varepsilon$
  exists since the compact containment condition for
  $\widetilde{\mathcal Z}^n$ does not hold.) For all $\delta>0$, let
  $K_\delta\subseteq E$ be compact and such that
  \begin{align*}
    \sup_{n=1,2,...} \mathbf P(\widetilde R_t^n \notin K_\delta \text{ for some
    }0\leq t\leq T) < \delta.
  \end{align*}
  For $\delta>0$ and $\varepsilon>0$ as above set
  $$ L_\delta := \{\mu\in\mathcal P(E): \mu(K^c_{\delta^2}) < \delta\}, \qquad
  L := \bigcap_{n=1}^\infty L_{\varepsilon 2^{-n}}.$$ Then, the
  closure of $L$ is a compact subset of $\mathcal P(E)$ by
  Prohorov's Theorem and by~\eqref{eq:1} there exist random times
  $\tau_k$, bounded by $T$ such that
  \begin{align} \label{eq:2} \sup_{k=1,2,...} \mathbf P(\widetilde
    Z_{\tau_k}^n \notin L) > \varepsilon.
  \end{align}
  Clearly, there must be $\ell\in\mathbb N$ such that
  \begin{align} \label{eq:3} \sup_{k=1,2,...} \mathbf P(\widetilde
    Z_{\tau_k}^n \notin L_{\varepsilon 2^{-\ell}}) > \varepsilon
    2^{-\ell}
  \end{align}
  (since otherwise \eqref{eq:2} cannot hold). Now we have by Markov's
  inequality that
  \begin{align*}
    \varepsilon 2^{-\ell} & < \sup_{k=1,2,...} \mathbf P(\widetilde
    Z_{\tau_k}^n(K^c_{\varepsilon^2 4^{-\ell}}) > \varepsilon
    2^{-\ell}) \leq \sup_{k=1,2,...} \tfrac{1}{\varepsilon 2^{-\ell}}
    \mathbf E[\langle \widetilde Z_{\tau_k}^n,
    1_{K^c_{\varepsilon^2 4^{-\ell}}}\rangle] \\
    & = \sup_{k=1,2,...} \tfrac{1}{\varepsilon 2^{-\ell}} \mathbf
    E[1_{K^c_{\varepsilon^2 4^{-\ell}}}(\widetilde R_{\tau_k}^n)] =
    \sup_{k=1,2,...}  \tfrac{1}{\varepsilon 2^{-\ell}} \mathbf
    P(\widetilde R_{\tau_k}^n \notin K_{\varepsilon^2 4^{-\ell}}) \\ &
    \leq \sup_{k=1,2,...}  \tfrac{1}{\varepsilon 2^{-\ell}} \mathbf
    P(\widetilde R_t^n \notin K_{\varepsilon^2 4^{-\ell}} \text{ for
      some }0\leq t\leq T) \leq \varepsilon 2^{-\ell},
  \end{align*}
  a contradiction.
\end{proof}

\begin{lemma}[Martingale convergence\label{l3}]
  Consider the same situation as in Theorem~\ref{T1} with
  $\nu=\delta_x$ for some $x\in E$. Let $\varphi \in \Pi$ and
  $\varphi_n \in \mathcal B(E)$ such that $||\varphi_n - \varphi||
  \xrightarrow{n\to\infty} 0$. For $f_n(z) := \langle z,
  \varphi_n\rangle$, consider the mean-zero martingale $\mathcal
  M^{n,\varphi_n} = (M^{n,\varphi_n}_t)_{t\geq 0}$, given by
  \begin{align*}
    M^{n,\varphi_n}_t &:= f_n(Z^n_{[nt]})-f_n(Z^n_0)-\sum_{k=1}^{[nt]}
    \mathbf E [f_n(Z^n_k) - f_n(Z^n_{k-1})|Z^n_{k-1}].
  \end{align*}
  Then, $\mathcal M^{n,\varphi_n}\xRightarrow{n\to\infty}0$ and
  $f_n(\widetilde{\mathcal Z}^n) \xRightarrow{n\to\infty} \mathbf
  E[\varphi(\mathcal R)]$. In particular, $f_n(\widetilde{\mathcal
    Z}^n))_{n=1,2,...}$ is tight.
\end{lemma}

\begin{proof}
  We start by reformulating, using~\eqref{eq:112},
  \begin{equation}
    \label{eq:133}
    \begin{aligned}
      M^{n,\varphi_n}_t & = \langle Z^n_{[nt]},\varphi_n\rangle -
      \langle Z^n_0, \varphi_n\rangle -\sum_{k=1}^{[nt]}
      \mathbf E[\langle Z^n_k, \varphi_n\rangle -\langle Z^n_{k-1}, \varphi_n\rangle|Z^n_{k-1}]\\
      &= \mathbf E[\varphi_n(R^n_{[nt]})|Z^n_{[nt]}]-\mathbf
      E[\varphi_n(R^n_0)|Z^n_0] - \int_{0}^{t-1/n} n \cdot \mathbf
      E\big[ \mathbf
      E[\varphi_n(R^n_{[ns]+1})-\varphi_n(R^n_{[ns]})]|Z^n_{[ns]}\big]ds \\
      & = \mathbf E[\varphi_n(\widetilde R^n_{t})|\widetilde
      Z^n_{t}]-\mathbf E[\varphi_n(\widetilde R^n_0)|\widetilde
      Z^n_0]- \int_{0}^{t-1/n} n \cdot \mathbf E\big[
      \varphi_n(\widetilde R^n_{s+1/n})-\varphi_n(\widetilde
      R^n_{s})|\widetilde Z^n_{s}\big]ds.
    \end{aligned}
  \end{equation}
  We now show that $\mathcal
  M^{n,\varphi_n}\xRightarrow{n\to\infty}0$. From Lemma~\ref{l:convZ},
  we already know that $\widetilde Z_t^n \xRightarrow{n\to\infty}
  \mathcal L(R_t)$. We complement this by showing that (note that the
  right hand side is deterministic) for all $s\geq 0$
  \begin{align*}
    n \cdot \mathbf E\big[ \varphi_n(\widetilde
    R^n_{s+1/n})-\varphi_n(\widetilde R^n_{s})|\widetilde Z^n_{s}\big]
    \xRightarrow{n\to\infty} \mathbf E\big[ G_{{\mathcal R}}
    \varphi(R_s)\big].
  \end{align*}
  Indeed,
  \begin{equation}
    \label{eq:345}
    \begin{aligned}
      \mathbf E\Big[\Big| & n \cdot \mathbf E\big[
      \varphi_n(\widetilde R^n_{s+1/n})-\varphi_n(\widetilde
      R^n_{s})|\widetilde Z^n_{s}\big]- \mathbf E[G_{\mathcal R}
      \varphi(R_s)]\Big|\Big] \\ & \leq \mathbf E\Big[\Big| n \cdot
      \mathbf E\big[ \varphi_n(\widetilde R^n_{s+1/n}) -
      \varphi_n(\widetilde R^n_{s})|\widetilde Z^n_{s}\big] - \mathbf
      E\big[ G_{{\mathcal R}}\varphi(\widetilde R^n_{s})|\widetilde
      Z^n_{s}\big]\Big|\Big] \\ & \qquad + \mathbf E\Big[\Big|\mathbf
      E\big[G_{{\mathcal R}}\varphi(\widetilde R^n_{s})-\mathbf E\big[
      G_{{\mathcal R}}\varphi(\widetilde R^n_{s})\big]
      \big| \widetilde Z^n_{s}\big] \Big|\Big] \\
      & \qquad \qquad + \Big|\mathbf E\big[ G_{{\mathcal
          R}}\varphi(\widetilde R^n_{s})\big]- \mathbf E[G_{\mathcal
        R} \varphi(R_s)]\Big| \xrightarrow{n\to\infty} 0
    \end{aligned}
  \end{equation}
  in probability, by \eqref{eq:T12}, Lemma~\ref{l:convZ} (which shows
  that the limit of $\widetilde Z_s^n$ is deterministic and hence the
  second to last line in~\eqref{eq:345} converges to~0), and weak
  convergence $\widetilde{\mathcal R}^n \xRightarrow{n\to\infty}
  \mathcal R$. For every $t\geq 0$, we now have that
  \begin{align}\label{eq:123}
    M^{n,\varphi_n}_t \xRightarrow{n\to\infty} \mathbf E[\varphi(R_t)]
    - \int_0^t \mathbf E[G_{\mathcal R}\varphi(R_s)]ds = 0.
  \end{align}
  Hence, we can write by Doob's inequality
  \begin{align}\label{eq:346}
    \mathbf P(\sup_{0\leq s\leq t}|M^{n,\varphi_n}_s|>\varepsilon)
    \leq \tfrac{1}{\varepsilon} \mathbf E[|M^{n,\varphi_n}_t|]
    \xrightarrow{n\to\infty} 0,
  \end{align}
  since $M_t^{n,\varphi}$ is bounded in $n$ and convergence
  in~\eqref{eq:123} also holds in probability. Then,
  using~\eqref{eq:133},
  \begin{align*}
    \mathbf P( & \sup_{0\leq s\leq t}|\langle \widetilde Z_s^n,
    \varphi_n\rangle - \mathbf E[\varphi(R_s)]| >3\varepsilon) \\ &
    \leq \mathbf P(\sup_{0\leq s\leq
      t}|M^{n,\varphi_n}_s|>\varepsilon) \\ & \qquad + \mathbf
    P\Big(\int_{0}^{t-1/n} \Big|n \cdot \mathbf E\big[
    \varphi_n(\widetilde R^n_{s+1/n})-\varphi_n(\widetilde
    R^n_{s})|\widetilde Z^n_{s}\big] ds- \mathbf E\big[ G_{{\mathcal
        R}} \varphi(R_s)\big]\Big|ds> \varepsilon\Big) \\ & \qquad
    \qquad + \mathbf P\Big(\int_{t-1/n}^t \Big|\mathbf E\big[
    G_{{\mathcal R}} \varphi(R_s)\big]\Big|ds> \varepsilon\Big) \\ &
    \xrightarrow{n\to\infty} 0
  \end{align*}
  by \eqref{eq:346} and \eqref{eq:345} .
\end{proof}

\begin{appendix}

\section{Random probability measures}
In the following, $(E,r)$ is a complete and separable metric space and
$\mathcal P(E)$ is the set of probability measures on (the Borel
$\sigma$-algebra of) $E$, equipped with the topology of weak
convergence. We will state some results about random measures.

\begin{definition}[First two moment measures]
  For a random variable $Z$, taking values in $\mathcal P(E)$, and
  $k=1,2,...$, there is a uniquely determined measure $\mu^{(k)}$ on
  $\mathcal B(E^k)$ such that
  $$ \mathbf E[Z(A_1)\cdots Z(A_k)] = \mu^{(k)}(A_1\times \cdots \times A_k)$$
  for $A_1,...,A_k\in\mathcal B(E)$. This is called the $k$th moment
  measure. Equivalently, $\mu^{(k)}$ is the unique measure such that
  $\mathbf E[\langle Z, \varphi_1\rangle \cdots \langle Z,
  \varphi_k\rangle] = \langle \mu^{(k)}, \varphi_1 \cdots
  \varphi_k\rangle$, where $\langle .,.\rangle$ denotes integration.
\end{definition}

\begin{lemma}[Characterisation of deterministic random
  measures\label{l:detMM}]
  Let $Z$ be a random variable taking values in $\mathcal P(E)$ with
  the first two moment measures $\mu:=\mu^{(1)}$ and $\mu^{(2)}$. Then
  the following assertions are equivalent:
  \begin{enumerate}
  \item There is $\nu \in \mathcal P(E)$ with $Z=\nu$, almost surely.
  \item The second moment measure has product-form, i.e.\ $\mu^{(2)} =
    \mu\otimes \mu$ (which is equivalent to
    $$ \mathbf E[\langle Z, \varphi_1\rangle\cdot \langle Z, 
    \varphi_2\rangle] = \langle \mu, \varphi_1\rangle \cdot \langle
    \mu, \varphi_2\rangle $$ for all $\varphi_1, \varphi_2 \in
    \mathcal C_b(E)$). (This is in fact equivalent to $ \mathbf
    E[\langle Z, \varphi\rangle^2] = \langle \mu, \varphi\rangle^2 $
    for all $\varphi \in \mathcal C_b(E)$).
  \end{enumerate}
  In either case, $\mu=\nu$.
\end{lemma}

\begin{proof}
  1.$\Rightarrow$2.: This is clear since we have $\mathbf E[Z(A)] =
  \nu(A)$, i.e.\ $\mu = \nu$. Moreover, $\mathbf E[Z(A_1) Z(A_2)] =
  \nu(A_1) \nu(A_2) = \mu(A_1)\mu(A_2) = \mu \otimes \mu (A_1\times
  A_2).$

  2.$\Rightarrow$1.: Since the second moment-measure has product form,
  for any measurable $A\subseteq E$, $\mathbf V[Z(A)] = \mathbf
  E[Z(A)Z(A)] - \mathbf E[Z(A)]^2 = \mu^{(2)}(A\times A) - (\mu(A))^2
  = 0$, i.e.\ the random variable $Z(A)$ has zero variance and
  therefore is deterministic. In particular, $Z(A) = \mathbf E[Z(A)] =
  \mu(A)$ and the assertions follows with $\nu=\mu$.
\end{proof}

\noindent
We end this appendix by recalling Jakubowski's tightness criterion
from \cite{Jakubowksi1986}; see also \cite{Dawson1993}, Theorem 3.6.4.

\begin{proposition}[Jakubowski's tightness criterion\label{P:jaku}]
  Assume the family $\Pi \subseteq \mathcal C_b(E)$ is a vector space
  that separates points. A sequence $\mathcal Z^1 = (Z_t^1)_{t\geq
    0}$, $\mathcal Z^2 = (Z_t^2)_{t\geq 0}$,... of $\mathcal
  P(E)$-valued processes with \cadlag-paths is tight if the following
  holds:
  \begin{enumerate}
  \item $(\mathcal Z^n)_{n=1,2,...}$ satisfies the compact containment
    condition.
  \item For every $f\in\Pi$, the sequence $(f(\mathcal
    Z^n))_{n=1,2,...}$ with $f(\mathcal Z^n) = (f(Z_t^n))_{t\geq 0}$
    is tight.
  \end{enumerate}
\end{proposition}

\end{appendix}

\subsubsection*{Acknowledgments}
This research was supported by the DFG through grant Pf-672/5-1.

\bibliographystyle{alpha}


\end{document}